\documentclass[11pt, twoside]{article}
\usepackage{amssymb, amsmath, amsthm}
\usepackage[left=27mm, right=27mm, bottom=30mm]{geometry}
\usepackage{inputenc}
\usepackage{fancyhdr}
\usepackage{tikz}
\usetikzlibrary{positioning}
\usepackage{titling}
\usepackage{hyperref}
\predate{}
\postdate{}
\usepackage{lipsum}
\newtheorem{theorem}{Theorem}
\newtheorem{lemma}[theorem]{Lemma}
\newtheorem{proposition}[theorem]{Proposition}

\newtheorem{question}[theorem]{Question}

\usepackage{titling}
\usepackage{xcolor}
\usepackage{placeins}
\usepackage{tikz}
\usepackage{verbatim}
\usepackage{titling}
\usepackage[T1]{fontenc}
\usepackage{currvita}
\usepackage{bbm}
\usepackage{enumitem}

\theoremstyle{definition}

\theoremstyle{remark}

\newcommand{\N}{\mathbb{N}}
\begin{document}
\newcommand{\Addresses}{
\bigskip
\footnotesize

\medskip

\noindent Tom\'a\v s~Fl\'idr, \textsc{Peterhouse, University of Cambridge, CB2 1RD, UK.}\par\noindent\nopagebreak\textit{Email address: }\texttt{tf388@cam.ac.uk}

\medskip

\noindent Maria-Romina~Ivan, \textsc{Department of Pure Mathematics and Mathematical Statistics, Centre for Mathematical Sciences, Wilberforce Road, Cambridge, CB3 0WB, UK,} and\\\textsc{Department of Mathematics, Stanford University, 450 Jane Stanford Way, CA 94304, USA.}\par\noindent\nopagebreak\textit{Email addresses: }\texttt{mri25@dpmms.cam.ac.uk, m.r.ivan@stanford.edu}}

\pagestyle{fancy}
\fancyhf{}
\fancyhead [LE, RO] {\thepage}
\fancyhead [CE] {TOM\'A\v S FL\'IDR AND MARIA-ROMINA IVAN}
\fancyhead [CO] {A COP-WIN GRAPH WITH MAXIMUM CAPTURE TIME $\omega$}
\renewcommand{\headrulewidth}{0pt}
\renewcommand{\l}{\rule{6em}{1pt}\ }
\title{\Large{\textbf{A COP-WIN GRAPH WITH MAXIMUM CAPTURE TIME $\omega$}}}
\author{TOM\'A\v S FL\'IDR AND MARIA-ROMINA IVAN}
\date{ }
\maketitle
\begin{abstract}
The game of cops and robbers is played on a fixed (finite or infinite) graph $G$. The cop chooses his starting position, then the robber chooses his. After that, they take turns and move to adjacent vertices, or stay at their current vertex, with the cop moving first. The game finishes if the cop lands on the robber's vertex. In that case we say that the cop wins, while if the robber is never caught then we say that the robber wins. The graph $G$ is called cop-win if the cop has a winning strategy. 

In this paper we construct an infinite cop-win graph in which, for any two given starting positions of the cop and the robber, we can name in advance a finite time in which the cop can capture the robber, but these finite times are not bounded above. This shows that this graph has maximum capture time (CR-ordinal) $\omega$, disproving a conjecture of Bonato, Gordinowicz and Hahn that no such graph should exist.
\end{abstract}
\section{Introduction}
Let $G$ be a fixed graph. Two players, the cop and the robber, each pick a starting vertex, with the cop picking first. Then they move alternately, with the cop moving first: at each turn the player moves to an adjacent vertex or does not move. The game is won by the cop if he lands on the robber. We say that $G$ is \textit{cop-win} if the cop has a winning strategy. Needless to say, if the graph is not connected then this game is a rather trivial robber win, so we assume from now on that all graphs are connected.

The finite cop-win graphs were characterised by Nowakowski and Winkler \cite{NW}. It is an easy exercise to see that if the graph contains a dominated vertex, say $x$, then $G$ is cop-win if and only if $G\setminus x$ is cop-win. (Here as usual we say that a vertex $y$ \textit{dominates} a vertex $x$ if the set of $x$ and all neighbors of $x$ is contained in the set of $y$ and all neighbors of $y$.) It is also easy to see that if no vertex is dominated then the robber has a winning strategy, so that $G$ is not cop-win -- on each turn, the robber moves to a vertex not adjacent to the cop. Putting these together, we see that a finite graph $G$ is cop-win if and only if it is constructible, meaning that it can be built up from the one-point graph by repeatedly adding dominated vertices. There are multiple variations of this game for finite graphs, such as allowing more than one cop to play. We direct the interested reader to the book of Bonato and Nowakowski \cite{BN} for general background and a wealth of other results in the finite case.

We now turn to infinite graphs, where the game of cops and robbers has the exact same rules as before. In this case, the situation is much more unclear. For example, an infinite ray is constructible, but it is trivially not a cop-win graph. Conversely, there exist cop-win graphs that are not constructible -- for this, and some other phenomena, see Ivan, Leader and Walters \cite{ILW}. We mention that currently, unlike the finite graphs set-up, there is no characterization for infinite cop-win graphs.

From now on, all the graphs will be assumed to be cop-win. For a finite graph, we can define \textit{the capture time} to be the length of the game, assuming both players play optimally. That means that the cop chooses his starting position as favorably as possible, and so does the robber. After that, every cop's next move will be to minimise the time he needs to capture the robber, and every robber's next move will be to escape for as long as possible. For example, if $G=P_{2n+1}$, the capture time is $n$. How can we generalise this notion for infinite graphs?

Let $G$ be a finite cop-win graph. We denote by $\eta(u,v)$ the time it will take a cop at $v$ to capture a robber at $u$, with the robber moving first. Let $\eta(v)=\sup_{u\in V(G)}\eta(u,v)$, or in words, how long the game lasts if the cop starts at $v$. Therefore, the capture time, denoted by $\eta(G)$ is $\min_{v\in V(G)}\eta(v)$. Another quantity we will be interested in is \textit{the maximum capture time}, denoted by $\rho(G)$. This is equal to $\sup_{u,v\in V(G)}\eta(u,v)$, or the maximum length the game can possibly last.

We mention that, once we go to the realm of infinite graphs, these quantities are not necessarily all finite, and so they can no longer be precisely interpreted as `times', or `number of moves'. However, all these notions generalise nicely if one works with ordinals. For the interested reader, we will briefly discuss this in Section 3. 

One of the main questions is what `times', hence ordinals, can the maximum capture time, $\rho(G)$ be? Such ordinals are known as \textit{CR-ordinals}. By considering finite paths, every finite ordinal is a CR-ordinal. Bonato, Gordinowicz and Hahn \cite{BGH} showed that all ordinals in $\{\omega\cdot i + (i + j): i, j <\omega\}\cup\{\alpha +\omega: \alpha\text{ is a limit ordinal}\}$ are CR-ordinals, and conjectured that in fact these are all of them. In particular, they conjectured that $\omega$ is not a CR-ordinal.

In this paper we disprove this conjecture by constructing a cop-win graph such that, for any given starting positions of the cop and the robber, one can name in advance a finite time in which the game terminates, but these times are unbounded. Moreover, this property is independent of whether the robber moves first, or the cop moves first. This indeed shows that the maximum capture time of this graph is $\omega$.

\section{The main result}
In this section we give an example of a cop-win graph with finite, but unbounded capture times, as discussed above. 

The intuition behind our example is that we want to be able to, on one hand, prolong the game arbitrarily, but still be able to name a time, given the players initial positions, even before they made any move. As such, no vertex should be `unpredictable', and having two quantities, one to decrease and hence insure the game will end, and one to prolong the game arbitrarily, might be desirable. It turns out that, to give the cop a chance of catching the robber, we need to introduce some `highways' for the cop to move along (these will be the two axes). Of course, the robber can also use these highways, but fortunately it will turn out that this does not help the robber as much as it helps the cop.

Based on this idea, we consider the $\N\times\N$ grid, and edges going up and left, and down and right.  Once the robber is pushed into one of the axes, which can be accessed by the cop in one move, then, by making any two vertices on an axis adjacent, he is immediately captured. 

Let $\mathcal G$ be the graph with vertex set $V=\N\times\N \setminus \{(0,0)\}$ and edge set $E$, where $\{(a,b),(c,d)\}\in E$ if and only if either $a=c=0$, or $b=d=0$, or $a<c$ and $b>d$, or $a>c$ and $b<d$. In other words, any two vertices on the $x$-axis are connected, any two vertices on the $y$-axis are connected, and, additionally, any vertex is connected to any vertex below it and to its right, as well as to any vertex above it and to its left. Note also that the graph is symmetric with respect to the reflection in the line $x=y$, i.e. swapping the axes. The picture below shows the neighbors of $(2,0)$ and $(4,4)$.
\begin{center}
\includegraphics[width=9cm]{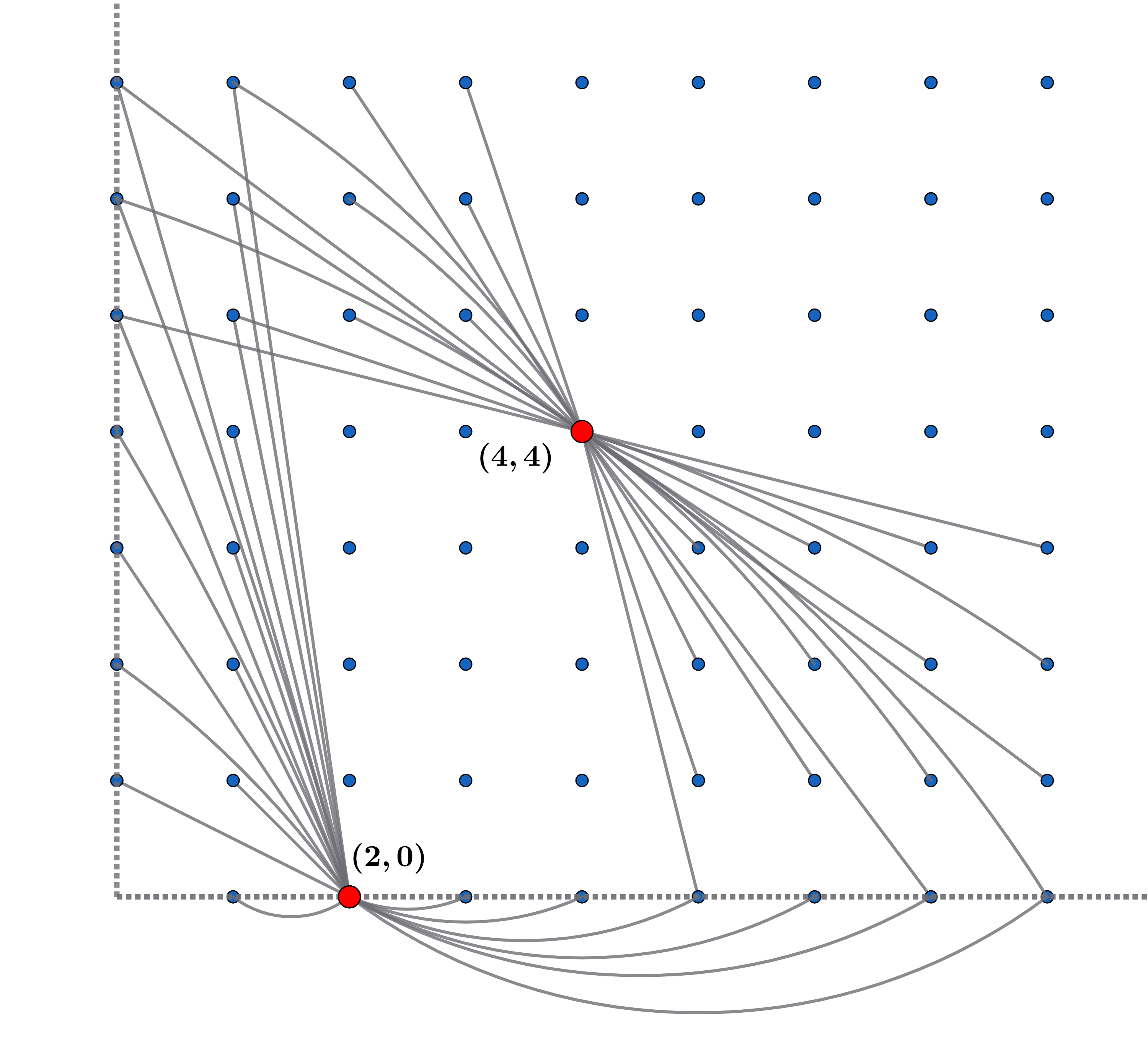}
\end{center}
\begin{theorem}
\label{main2}
The graph $\mathcal G$ is cop-win. Furthermore, for any two vertices  $u,v$ of $\mathcal G$ there exists $n\in\N$ such that if the cop starts at $v$ and robber at $u$, the cop can catch the robber in at most $n$ turns. However, there is no $n\in\N$ such that the above is true for all initial positions.
\end{theorem}
The proof is an immediate corollary of the following theorem.
\begin{theorem}
\label{main}
The graph $\mathcal G$ is cop-win. Furthermore, for any two vertices  $u,v$ of $\mathcal G$ there exists $n\in\N$ such that if the cop is at $v$ and robber at $u$, with the robber having to make the next move, the cop can catch the robber in at most $n$ turns. However, there is no $n\in\N$ such that the above is true for all initial positions.
\end{theorem}
\begin{proof}
We first show that $\mathcal G$ is a cop-win graph. 

We start by noticing that if the robber moves at a vertex of the form $(a,0)$ or $(a,1)$ for some $a\geq 1$, then he is captured in at most two moves. Indeed, from any vertex, the cop can move to $(c,0)$ for some $c>a$. Then, the cop's vertex is adjacent to the robber's vertex. If the robber is at $(a,0)$ or $(a,1)$, he can either move to a vertex on the $x$-axis, and so he is immediately captured, or up and to the left. But then he is also immediately captured as any such vertex is adjacent to $(c,0)$. By symmetry, the same is true if the robber is at $(0,a)$ or $(1,a)$ for some $a\geq1$.

Let $(a,b)$ be the robber's vertex, and $(c,d)$ the cop's vertex, and it is the cop's turn to make a move. If $a=0$ or $b=0$, we already showed that the robber is caught in at most two moves, so we may assume that $a>0$ and $b>0$. Next, the cop moves to $(a+c+1,0)$. If the robber does not want to be captured, he must go down and to the right. We now noticed that with one move we have forced the robber to strictly decrease his $y$-coordinate. Continuing like this, he will reach the $x$-axis in at most $b$ steps, where he will be captured in at most two steps. This shows that $\mathcal G$ is cop-win.

Next, suppose that the robber is at $(a,b)$ and the cop is at $(c,d)$, and it is the robber's turn to make a move. We will show that he will be captured in at most $\max\{a,b\}+3$ moves. Suppose that the robber moves at $(a',b')$ (which could be equal to $(a,b)$). If $a'=0$ or $b'=0$, then he is captured in at most 2 moves, so we may assume that $a'>0$ and $b'>0$. Since $(a,b)$ is adjacent to $(a',b')$ (or equal), we may assume without loss of generality that $b'\leq b$. Now, the cop moves to $(a'+c+1,0)$. Notice that we showed above that in this situation, the robber will be captured in at most $b'\leq b$ moves, which finishes the claim.

Lastly, we have to show that for any $n\geq1$, there exist starting positions for the cop and the robber, with the robber moving first, such that the game lasts at least $n$ steps. Let the cop be at $(c,d)$ and then the robber at $(a,b)$ such that $a,b\geq n+1$. If $(a,b)$ and $(c,d)$ are not adjacent, then the robber just stays at $(a,b)$. If they are adjacent, then, we may assume without loss of generality that $a<c$ and $b>d$, and the robber moves to $(c+1,b-1)$, which is not adjacent to any vertex $(c,d)$, the cop's vertex, is adjacent to. Continuing like this, every step, the robber decrease one coordinate by at most 1, and not get captured until he reaches one of the axes, or $(a_0,1)$ or $(1,b_0)$ for some $a_0,b_0\geq 1$. Hence he can survive for at least $n$ moves, finishing the proof.
\end{proof}
\section{Generalised capture time}

In the case where the capture time is finite, we can interpret it as `length of the game'. However, in general, using ordinals for capture times is very natural. We make this precise here. 

Let $G$ be a graph of cardinality $\aleph_{\beta}$, where $\beta\geq 0$ is an ordinal. For any vertex $v\in V(G)$ we denote by $N[v]$ its closed neighborhood (its neighbors, together with $v$). Set $\omega(G)=\omega_{\beta+1}$, and define the relations $\{\leq\alpha\}_{\alpha<\omega(G)}$ on $V(G)$ as follows. If $u = v$, then $u\leq_0 v$. Also, $u\leq_{\alpha} v$ if for all $x\in N[u]$, there exists $y\in N[v]$ such that $x\leq_{\gamma} y$ for some $\gamma<\alpha$.

Now, for any two vertices $u, v$, we define $\eta(u, v)=\alpha $, where $\alpha$ is the minimum ordinal for which $u\leq_\alpha v$ holds. We observe that if $\eta(u, v)$ is finite, then it is precisely equal to the time it takes the cop at $v$ to catch the robber at $u$, assuming both play optimally, and the robber moves first. Define $\eta(v)$ to be the minimum ordinal $\alpha$ such that $u\leq_{\alpha} v$ holds for all $u\in V(G)$. Finally, define $\eta(G) = \min_{v\in V(G)}\eta(v)$. One can check that if $\eta(G)$ is finite, then it is precisely the capture time of the cop-win graph $G$. Let also $\rho(G) =\sup_{v\in V(G)}\eta(v)$. In the finite case, $\rho(G)$ is the maximum capture time over all initial positions of the cop. 

We make a few observations. Suppose $G$ is a cop-win graph such that $\eta(u,v)$ is always finite, but $\sup_{u,v\in V(G)}\eta(u,v)=\omega$, then, by definition, $\rho(G)=\omega$. But $\eta(u,v)$ being finite, say $n$, means that a cop at $v$ will \textit{always} capture a robber at $u$ in at most $n$ moves, with the robber moving first. We have showed in Theorem~\ref{main} that our graph $\mathcal G$ is cop-win, and, for any initial positions of the cop and the robber (with the robber moving first), one can name a finite capture time, \textit{before} any player made any move, but these times are unbounded over all starting positions. This gives us the following.

\begin{theorem}
Let $\mathcal G$ be the graph constructed above. Then $\rho(G)=\omega$. In other words, $\omega$ is a CR-ordinal.
\end{theorem}
\section{Further remarks}
Surprisingly, we also have the following.
\begin{proposition}
The graph $\mathcal G$ is constructible.
\end{proposition}
\begin{proof}
We start with $(1,0)$, then add $(0,1)$, then $(0,2)$ being dominated by $(1,0)$, then $(1,1)$ being dominated by $(0,2)$, then $(2,0)$, being dominated by $(2,0)$. We continue like this by moving along the axes once and then on the diagonal with slope $-1$ until we reach the other axis.
\end{proof}
We end with a few observation about cop-win graphs with CR-ordinal $\omega$ such that $\eta(u,v)\neq\omega$ for any two vertices $u$ and $v$. We call such graphs \textit{limit} graphs.
\begin{lemma}
Let $G$ be a limit graph with infinite diameter. Then all vertices of $G$ have infinite degree.
\end{lemma}
\begin{proof}
Assume for a contradiction that $x\in V(G)$ is a vertex with finite degree. Since $G$ has infinite diameter, the lengths of the shortest paths from $x$ to the rest of the vertices of $G$ is unbounded. Since $x$ has a finite number of neighbors, there exists $y$ adjacent to $x$ such that infinitely many of these paths that start at $x$ must go through $y$ next.

Now assume that the cop is at $x$ and the robber is at $y$. For any $n\in\N$, the robber can follow a shortest path from $x$ of length at least $n$ and hence survive for at least $n$ turns. Therefore, the cop cannot guarantee to capture the robber in $n$ moves for any $n\in\N$, a contradiction.
\end{proof}

\begin{lemma}
Let $G$ be a limit graph, and $s\in V(G)$ with degree at least 2. Then the cop's winning strategy cannot have the property that, given his position and regardless of the robber's position, he commits to a shortest path to $s$ to follow, until he either intercepts the robber, or reaches $s$.
\end{lemma}
\begin{proof}
Suppose first that $G$ has infinite diameter. Then, for every $n\in\N$ there exists $x_n\in V(G)$ such that $d(s,x_n)\geq n$. Now, for all $n\in\N$ fix such a $x_n$ and a shortest path from $s$ to $x_n$, say $P_n$. Let also $y_1$ be the first vertex after $s$ in $P_1$. If infinitely many $P_n$ go through $y_1$, then put the cop at some other neighbor of $s$, and the robber at $s$. Then the robber can choose any long path going through $y_1$ to follow and survive for arbitrarily long time. If only finitely many of these paths go through $y_1$, then put the cop at $y_1$ and the robber at $s$. The robber can now choose all the paths that do not go thought $y_1$ and again, survive for an arbitrary long time, a contradiction.

Therefore $G$ has finite diameter. Assume for a contradiction that the capture times (for a cop) from $s$ to the rest of the vertices are arbitrarily large. Let $x$ be a vertex such that $d(s,x)$ is maximal, and $x_n$ a vertex such that the capture time for a cop at $s$ and a robber at $x_n$ is at least $n$.

Suppose first that $d(s,x)$ is odd. We put the cop at $x$ and the robber at $s$. Then, for any $n\in\N$ the robber can move to a vertex with capture time from $s$ at least $n$, and stay there until the cop reaches $s$. This contradicts the assumption that the capture time from $x$ to $s$ is finite. Note that no matter what shortest path to $x_n$ the robber might have chosen, and no matter what shortest path to $s$ the cop might have chosen, they could not have intersected as the distance from $s$ to $x$ is odd.

Assume now that $d(x,s)$ is even, say $2k$, and the cop is at $x$. Let $P$ be the shortest path from $x$ to $s$ that the cop commits to. If the robber is at $s$ and tries to get to $x_n$, he can only be intercepted if his path goes through the middle of $P$. If this does not happen for infinitely many $n$, then the robber can still get to infinitely many $x_n$, a contradiction. Therefore, for infinitely many $n$, the shortest paths from $s$ to $x_n$ go through the middle vertex of $P$, call it $y$. In this case, we put the robber at $y$, who can now get to infinitely many $x_n$ before the cop gets to $s$ via $P$, a contradiction.

Therefore, there exists $n_0$ such that a cop at $s$ will catch the robber anywhere in the graph in at most $n_0$ steps. Since $G$ has finite diameter, this implies that the maximum capture time is also finite, a contradiction.
\end{proof}

Notice that we have found a cop-win graph with maximal capture time $\omega$, however $\omega$ is not attained by any $\eta(u,v)$. It is therefore natural to ask the following.
\begin{question}
Does there exist a cop-win graph $G$ such that $\rho(G)=\omega$, and there exists $u,v\in V(G)$ such that $\eta(u,v)=\omega$? In other words, does there exist a cop-win graph such that for any two vertices one of the following is true, with the second case happening at least once: either we can name a finite end time in advance, or we cannot, in which case, for any first move the robber can make, the cop can make a move such that we can now name a finite end time for the new positions?  
\end{question}

\noindent{\textbf{Acknowledgment.} The first author would like to thank G-Research for generously funding his stay in Cambridge while undertaking this project.}

\Addresses

\begin{thebibliography}{99}
\bibitem{BN} A. Bonato and R. Nowakowski, The Game of Cops and Robbers on Graphs. \textit{American Mathematical Society} (2011), ISBN-13 978-0821853474.
\bibitem{BGH} A. Bonato, P. Gordinowicz and G. Hahn, Cops and Robbers ordinals of cop-win trees. \textit{Discrete Mathematics}, \textbf{340} (2017), 951 -- 956.
\bibitem{ILW} M.-R. Ivan, I. Leader and M. Walters, Constructible Graphs and Pursuit. \textit{Theoretical Computer Science}, \textbf{930} (2022), 196 -- 208.
\bibitem{NW} R. Nowakowski and P. Winkler, Vertex-to-vertex pursuit in a graph. \textit{Discrete Mathematics}, \textbf{43} (1983), 235--239.
\end{thebibliography}
\end{document}